\documentclass[11pt]{article}

\usepackage{fullpage}

\usepackage{float}

\usepackage{caption} 

\usepackage{graphicx} 

\usepackage{amsthm}

\usepackage{amsmath}

\usepackage{amssymb}
\usepackage{hyperref}
\usepackage{lineno}

\newtheorem{theorem}{Theorem}[section]
\newtheorem{corollary}[theorem]{Corollary}
\newtheorem{lemma}[theorem]{Lemma}
\newtheorem{definition}[theorem]{Definition}

\newtheorem{fact}[theorem]{Fact}
\newtheorem{claim}[theorem]{Claim}
\newtheorem{proposition}[theorem]{Proposition}
\begin{document}

	\title{Spanning trees in graphs without large bipartite holes }
	
	\author{Jie Han  \thanks{School of Mathematics and Statistics and Center for Applied Mathematics, Beijing Institute of Technology, Beijing, China. Email: {\tt han.jie@bit.edu.cn.}}\and
		 Jie Hu\thanks{Center for Combinatorics and LPMC, Nankai University, Tianjin, 300071, China. Email: {\tt hujie@nankai.edu.cn}. }\and
		 Lidan Ping \thanks{School of Mathematics, Shandong University, China. Email: {\tt p1999pinglidan@126.com}. }\and
		 Guanghui Wang \thanks{School of Mathematics, Shandong University, China. Email: {\tt ghwang@sdu.edu.cn}. Research supported by Natural Science
		 	Foundation of China (12231018) and Young Taishan Scholars probgram of Shandong Province( 201909001).}\and
		Yi Wang \thanks{Data Science Institute, Shandong University, China. Email: {\tt 202012034@mail.sdu.edu.cn}.}\and
		 Donglei Yang\thanks{Data Science Institute, Shandong University, China. Email: {\tt dlyang@sdu.edu.cn}. Supported by the China Postdoctoral Science Foundation (2021T140413), Natural Science Foundation of China (12101365) and Natural Science Foundation of Shandong Province (ZR2021QA029).} }
	\date{\today}
	\maketitle
	\begin{abstract}

We show that for any $\varepsilon>0$ and $\Delta\in\mathbb{N}$, there exists  $\alpha>0$  such that  for sufficiently large $n$, every $n$-vertex graph $G$ satisfying that $\delta(G)\geq\varepsilon n$ and $e(X, Y)>0$ for every pair of disjoint vertex sets $X, Y\subseteq V(G)$ of size $\alpha n$ contains all spanning trees with maximum degree at most $\Delta$. 
This strengthens a result of B\"ottcher et al.

	\end{abstract}

	\section{Introduction}	
	Determining the minimum degree condition for the  existence of spanning structures is a central problem in extremal graph theory. The first result of this direction is Dirac's theorem \cite{dirac1952some} in 1952 which states that an $n$-vertex graph $G$ with $\delta(G)\ge\frac{n}{2}$  contains a	Hamiltonian cycle.
	Koml\'os, S\'ark\"ozy and Szemer\'edi \cite{komlos1995proof} proved that  an $n$-vertex graph $G$ with $\delta(G)\geq(\frac{1}{2}+o(1))n$ contains a copy of every  bounded-degree spanning tree, and in \cite{komlos2001spanning},
	the result is extended to trees with maximum degree  $O(\frac{n}{\rm {log} n})$. 
	Another notable extension is the bandwidth theorem of B\"ottcher, Schacht and Taraz \cite{2009Proof}, which finds the (asymptotically) optimal minimum degree condition forcing spanning subgraphs with bounded chromatic number and sublinear bandwidth and  resolves a conjecture  of Bollob\'as and Koml\'os \cite{komlos_1999}. 
	
	Note that the extremal graphs in this type of results all have large independent sets,  which makes them far from being typical.
	Hence a natural project is to study how the degree conditions drop if we  forbid large independent sets from the host graph. Balogh, Molla and Sharifzadeh \cite{balogh2016triangle} initiated this study by proving if $G$ is an $n$-vertex graph with $\delta(G)\geq(\frac{1}{2}+o(1))n$ and $\alpha(G)=o(n)$, then $G$ contains a $K_{3}$-factor.
	This result is asymptotically optimal and 
	weakens the bound $\delta(G)\geq\frac{2}{3}n$ from the Corr\'adi--Hajnal theorem \cite{K1963On}.  Knierim and Su \cite{KNIERIM202160} resolved the case $K_{r}$-factor for $r\geq 4$
	by showing $\delta(G)\geq(\frac{r-2}{r}+o(1))n$. 
	Nenadov and Pehova \cite{nenadov2020ramsey} further generalised  this  into  $\ell$-independence number and this inspires several recent works 
	\cite{2111.10512,2203.02169,han2021ramsey}.
	
		However, just by excluding  large independent sets is not enough to guarantee the existence of large \emph{connected} subgraph. 
	The union of two disjoint copies of $K_{\frac{n}{2}}$ has independence number two, but
	it does not contain any connected subgraphs with more than $\frac{n}{2}$ vertices. 
	Hence  it is necessary to  impose stronger conditions to overcome this.  
The following notion of bipartite hole was introduced by McDiarmid and Yolov~\cite{2017Hamilton}  in the study of Hamilton cycles. 
	
	Given two disjoint vertex sets $A$ and $B$ in a graph $G$, we use $E_{G}(A,B)$ to denote the set of edges joining $A$ and $B$, and let $e_{G}(A,B)=\vert E_{G}(A,B) \vert$.

	\begin{definition}[\cite{2017Hamilton}]
		\rm{Given a graph $G$, an \textit{$(s,t)$-bipartite-hole} in $G$ consists of two disjoint sets $S,\ T\subseteq V(G)$ with $\vert S \vert=s$ and $\vert T \vert=t$ such that $e_{G}(S,T)=0$.
		We define the \textit{bipartite-hole number $\widetilde{\alpha}(G)$} to be the maximum integer $r$ such that $G$ contains an $(s,t)$-bipartite-hole for every pair of nonnegative integers $s$ and $t$ with $s+t=r$.}
	\end{definition}

	In this paper, we adopt a slightly different notion of bipartite-hole number due to Nenadov and Pehova~\cite{nenadov2020ramsey}. 
	The \emph{bipartite independence number} $\alpha^*(G)$ is the maximum integer $t$ such that $G$ contains a $(t,t)$-bipartite-hole.	
	It is clear from the definition that  $2\alpha^*(G)+1\geq\widetilde{\alpha}(G)\geq \alpha^*(G)+1$.
	McDiarmid and Yolov~\cite{2017Hamilton} showed that $\delta(G)\ge\widetilde{\alpha}(G)$ is enough to force $G$ to be Hamiltonian and this result is sharp.
	Recently, Kim, Kim and Liu~\cite{2020Tree} studied the decomposition of an almost regular graph $G$ with $\alpha^*(G)=o(n)$ into almost spanning trees of bounded maximum degree.
	
	The main result of this paper is the following. We denote by $\mathcal{T}(n,\Delta)$ the family of all trees on $n$ vertices with maximum degree at most $\Delta$.

		\begin{theorem}\label{1}
		For each $ \varepsilon>0$ and $\Delta \in \mathbb{N} $, there exists $\alpha=\alpha(\varepsilon,\Delta)>0$ such that the following holds for sufficiently large $n \in \mathbb{N}$. Every $n$-vertex graph $G$ with $\delta(G)\ge\varepsilon n$ and $\alpha^*(G)<\alpha n$ is $\mathcal{T}(n,\Delta)$-universal, that is, $G$ contains every $T\in \mathcal{T}(n,\Delta)$ as a subgraph.
	\end{theorem}
	
That is, for a graph $G$ with sublinear $\alpha^*(G)$, the minimum degree condition forcing bounded-degree spanning trees is almost sublinear. 
Another motivation of our result is its connection to the randomly perturbed graphs introduced by Bohman, Frieze and Martin~\cite{BFM}, where the host graph is obtained by adding random edges to a deterministic graph with minimum degree conditions.
Krivelevich, Kwan and Sudakov \cite{Krivelevich2017BOUNDED} showed that for any $\varepsilon>0$, $\Delta\in \mathbb{N}$ and $T\in \mathcal{T}(n,\Delta)$, if $G$ is an $n$-vertex graph with $\delta(G)\geq\varepsilon n$, then the randomly perturbed graph $G\cup G(n,\frac{C}{n})$ \textit{a.a.s.} contains $T$ as a subgraph, where $C$ depends only on $\varepsilon$ and $\Delta$.
They suggested that such a result can be improved to a universality result, that is, $G\cup G(n,\frac{C}{n})$ \textit{a.a.s.} contains all members of $\mathcal{T}(n,\Delta)$ simultaneously.
This is confirmed by B\"ottcher et al.~\cite{bottcher2019universality}.
Indeed, in their technical result which we state below, they replaced $G(n,\frac{C}{n})$ by a deterministic sparse expander graph, and thus get the universality part for free.

\begin{theorem}[\cite{bottcher2019universality}, Theorem 2]
\label{thm2}
For any $\alpha>0$ and integers $C\ge2$ and $\Delta>1$, there exist $\varepsilon>0$, $D_0$ and $n_0$ such that the following holds for any $D\ge D_0$ and $n\ge n_0$.
Let $G$ be an $n$-vertex graph satisfying the following two conditions:\\
1. $\Delta(G)\leq CD$,\\
2. $e(U, W)\ge \frac{D}{Cn}|U||W|$ for all sets $U, W\subseteq V(G)$ with $|U|,|W|\ge\varepsilon n$.\\
Suppose $G_\alpha$ is an $n$-vertex graph on the same vertex set and $\delta(G_\alpha)\ge\alpha n$.
Then $H:=G_\alpha\cup G$ is $\mathcal{T}(n,\Delta)$-universal.
\end{theorem}

Theorem~\ref{1} slightly improves upon Theorem~\ref{thm2}, where we do not need to distinguish two graphs (or equivalently the  maximum degree condition on the sparse graph $G$ is no longer needed).


\section{Proof strategy and Preliminaries}
\subsection{Notation}

For a graph $G=(V,E)$, let $v(G)=\vert V\vert$ and $e(G)=\vert E\vert$. 
Given a collection of subgraphs $\mathcal{F}=\left\{ F_{i}: i\in I\right\}$, let $V(\mathcal{F})=\cup_{i\in I}V(F_{i})$.
Given two vertex-disjoint graphs $G_1, G_2$, let $G_1\cup G_2$ be the union of $G_1$ and $G_2$.
For $U \subseteq V(G)$, let $G[U]$ be the induced subgraph of $G$ on $U$ and let $G-U$ be the induced graph after removing $U$, that is $G-U:=G[V\backslash U]$.
Given a vertex $v\in V(G)$ and $X,Y\subseteq V(G)$, denoted by $N_{X}(v)$  the set of neighbors of $v$ in $X$ and let $d_{X}(v):=\vert N_{X}(v)\vert $. 
The neighborhood of $X$ in $G$ is denoted by  $N_{G}(X)=(\cup_{v\in X}N(v))\backslash X$ and let $N_{Y}(X)=N_{G}(X)\cap Y$.
We omit the index $G$ if the graph is clear from the context.

For a graph $F$ on $[k]:=\left\{ 1,\ldots,k\right\}$, we say that $B$ is the \textit{$n$-blow-up} of $F$ if there exists a partition $V_{1},\ldots,V_{k}$ of $V(B)$ such that $\vert V_{1}\vert=\cdots=\vert V_{k} \vert =n$ and we have that $\{u,v\}\in E(B)$ if and only if $u\in V_{i}$ and $v\in V_{j}$ for some $\{i,j\}\in E(F)$.
Given a spanning subgraph $G$ of $B$, we call the sequence $V_1,\dots, V_k$ the \textit{parts} of $G$ and we define $\bar{\delta}(G):=\min_{\{i,j\}\in E(F)}\delta(G[V_i,V_{j}])$ where $G[V_i,V_{j}]$ is the bipartite subgraph of $G$ induced by the parts $V_i$ and $V_{j}$. A subset $R$ of $V(G)$ is \textit{balanced} if $\vert R\cap V_1\vert=\cdots=\vert R\cap V_k\vert$. In particular,
we say a subset of $V(G)$ or a subgraph of $G$  \textit{transversal} if it intersects each part in exactly one vertex.

For a path $P$, the \textit{length} of $P$ is the number of edges in $P$.
Given two vertices $x,y$, an \textit{$x,y$-path} is a path with ends $x$ and $y$.
Let $T$ be a tree and $T^{\prime}$ be obtained from $T$ by removing all leaves. 
A \textit{pendant star} is a maximal star centered at a  leaf of $T^{\prime}$, where the unique neighbor of the center inside $T^{\prime}$ is called the \textit{root} of the pendant star. 
A \textit{bare path} in $T$ is a path whose internal vertices have degree exactly two in $T$.
A \textit{caterpillar} in $T$ consists of a bare path in $T^{\prime}$ as the \textit{central path} with some (possibly empty) leaves attached to the internal vertices of the central path, where \emph{branch} vertices are the internal vertices attached with at least one leaf. 
The \textit{length} and \textit{ends} of the caterpillar  refer to the length and ends of its central path.

For two graphs $H$ and $G$, an \textit{embedding} $\varphi$ of $H$ in $G$ is an injective map $\varphi: V(H) \rightarrow V(G)$ such that $\left\{ v,w \right\}\in E(H)$  implies $\left\{ \varphi(v),\varphi(w) \right\}
\in E(G)$.	
For all integers $a, b$ with $a\leq b$, let $[a,b]:= \left\{ i\in \mathbb{Z}:a\leq i \leq b \right\}$ and $[a]:=\left\{ 1, 2,\dots, a \right\}$.	When we write $\alpha \ll \beta \ll \gamma$, we always mean that $\alpha, \beta, \gamma$ are constants in $(0, 1)$, and $\beta \ll \gamma$ means that there exists $\beta_{0}=\beta_{0}(\gamma)$ such that the subsequent arguments hold for all $0<\beta \leq \beta_{0}$. Hierarchies of other lengths are defined analogously. For the sake of clarity of presentation, we will sometimes omit floor and ceiling signs when they are not crucial. 

\subsection{Graph Expansion and Trees}
We will introduce some graph expansion properties to embed the trees.
\begin{definition}[\cite{Johannsen2013Expanders}]\label{99}
	\rm{Let $n \in \mathbb{N}$ and $d>0$. A graph $G$ is an \textit{$(n,d)$-expander} if $\vert G \vert=n$ and $G$ satisfies the following two conditions.\\
		1. $\vert N_{G} (X)\vert \geq d\vert X \vert $ for all sets $X \subseteq V(G)$ with $1\leq\vert X \vert<\lceil \frac{n}{2d} \rceil$.\\
		2. $e_{G}(X,Y)>0$ for all disjoint  $X,Y \subseteq V(G)$ with $\vert X \vert=\vert Y \vert=\lceil \frac{n}{2d} \rceil$}.
	
\end{definition}

In \cite{2010Embedding}, Krivelevich considered trees differently according to whether they contain many leaves or many disjoint bare paths.
\begin{lemma}[\cite{2010Embedding}]\label{61}
	For any integers $n,k>2$, a tree on $n$ vertices either has at least $\frac{n}{4k}$ leaves or a collection of at least $\frac{n}{4k}$ vertex-disjoint bare paths of length $k$.		
\end{lemma}

We will use the following corollary to divide $\mathcal{T}(n,\Delta)$ into trees with many  pendant stars and trees with many vertex-disjoint caterpillars. 
\begin{corollary}\label{2}
	For any integer $n,k>2$, a tree on $n$ vertices with maximum degree $\Delta$ either has at least $\frac{n}{4k\Delta}$ pendant stars or a collection of at least $\frac{n}{4k\Delta}$ vertex-disjoint caterpillars each of length $k$.
\end{corollary}

Suppose $T$ is an $n$-vertex tree with maximum degree at most $\Delta$ and $T^{\prime}$ is the subtree obtained by removing all leaves. Then it holds that $\vert T^{\prime} \vert+\vert T^{\prime} \vert(\Delta-1)\geq n$.
We can apply Lemma \ref{61} on $T^{\prime}$ to obtain at least $\frac{|T^\prime|}{4k}\ge\frac{n}{4k\Delta}$ pendant stars or a collection of at least $\frac{n}{4k\Delta}$ vertex-disjoint caterpillars of length $k$ in $T$.

In \cite{2001Tree}, Haxell extended a result of Friedman and Pippenger \cite{1987Expanding} and showed that one can embed every almost spanning tree with bounded maximum degree in a graph with strong expansion property. 
We will use the following result by Johannsen, Krivelevich and Samotij \cite{Johannsen2013Expanders}.

\begin{theorem}[\cite{Johannsen2013Expanders}]
	Let $n,\Delta \in \mathbb{N}$, $d \in  \mathbb{R}^+$ with $d \geq 2\Delta$ and $G$
	be an $(n,d)$-expander. Given any $T \in \mathcal{T}(n-4\Delta \lceil \frac{n}{2d} \rceil, \Delta)$, we can find a copy of $T$ in  $G$.  \label{9}
\end{theorem}

As shown by Johannsen, Krivelevich and Samotij \cite{Johannsen2013Expanders}, the following lemma is useful for attaching leaves onto certain vertices.

\begin{definition}[\cite{Johannsen2013Expanders}]\label{99}
	\rm{Given a bipartite graph $G=(A,B,E)$ with $\vert A\vert\leq\vert B\vert$ and a function $f: A\rightarrow \mathbb{N}$ with $\sum_{u\in A}f(u) = |B|$, an $f$-\emph{matching  from} $A$ \emph{into} $B$ is a collection of vertex-disjoint stars $\{S_u:u\in A\}$ in $G$ such that $S_u$ has $u$ as the center and exactly $f(u)$ leaves inside $B$}.		
\end{definition}

\begin{lemma}[\cite{Johannsen2013Expanders}]\label{11}
	Let $d,m\in \mathbb{N}$ and let $G$ be a graph. Suppose that two disjoint sets $U,W \subseteq V(G)$ satisfy the following three conditions:\\
	1. $\vert N_{G} (X) \cap W\vert \geq d\vert X \vert $ for all sets $X \subseteq U$ with $1\leq\vert X \vert\leq m$,\\
	2. $e_{G}(X,Y)>0$ for all $X \subseteq U$ and $Y\subseteq W$ with $\vert X \vert=\vert Y \vert\geq m$,\\
	3. $d_U(w) \geq m$ for all
	$w \in W$.\\
	Then for every $f: U\to\{1,\dots,d\}$ with $\sum_{u\in U}f(u)=\vert W\vert$, there exists an $f$-matching from $U$ into $W$.

\end{lemma}

\subsection{Proof Overview}

We give a brief outline of our proof here.  Similar in previous works~\cite{2010Embedding,Krivelevich2017BOUNDED,montgomery2019spanning}, we classify the trees and deal with them separately. 
Our classification is given in Corollary~\ref{2}, which refines Lemma~\ref{61}.
First, for the pendant star case, let $T_1$ be a subtree obtained  by deleting the centers and  leaves of pendant stars in $T$ (but not the roots).  
We randomly partition $V(G)$ into $V_{1}, V_{2}$ and $V_{3}$ and 
embed 
$T_1$ into $G[V_{1}]$ by Theorem \ref{9}.
Then we greedily embed most of the centers into $V_{2}$.
The $\alpha^*$ property guarantees that we are left with a small portion of centers and we shall embed them by the degree condition into $V_{3}$. 
Finally we use Lemma~\ref{11} to find a desired star-matching and complete the embedding.

For the caterpillar case, we shall embed a suitable subset of branch vertices into a random set with ``good" expansion properties.
In this way, we can greedily finish the last step of embedding leaves of caterpillars by a star-matching.
To embed this suitable branch vertex set, we control the length of the caterpillars and
build an auxiliary graph $G=(V_{1},\ldots,V_{k},E)$ which is a spanning subgraph of the $n$-blow-up of $C_k$. Then 
we transform the problem of embedding vertex-disjoint paths  into finding a transversal cycle-factor in $G$, as Lemma \ref{33} below.
Let $\alpha^*_{\rm b}(G)$ be the largest integer $s$ such that $G=(V_{1},\ldots,V_{k},E)$ contains an $(s,s)$-bipartite-hole $(S,T)$ where $S\subseteq V_i, T\subseteq V_{i+1}$ for some $i\in[k]$.

\begin{lemma}\label{33}
	Given a positive integer $k \in 4\mathbb{N}$ and a constant $\delta$ with $\delta>\frac{2}{k}$, there exists $\alpha>0$ such that the following holds for sufficiently large $n \in \mathbb{N} $. 
	Let $G=(V_{1},\ldots,V_{k},E)$ be  a spanning subgraph of the n-blow-up of $C_{k}$ with $\bar{\delta}(G)\geq \delta n$ and $\alpha^*_{\rm b}(G)<\alpha n$. Then $G$ has a transversal $C_{k}$-factor.
\end{lemma}

Although the minimum degree bound in Lemma~\ref{33} is not best possible and it only works for $k \in 4\mathbb{N}$, but it is enough for our purpose (indeed, we only need it work for large $k$).
We suspect that the (asymptotic) tight condition should be $(1+o(1))\frac nk$, given by the so-called space barrier.
Indeed, let $G$ be a (complete) $n$-blow-up of $C_k$ with parts labelled as $V_{1},\ldots,V_{k}$.
One can specify a set $U_i$ of size $n/k-1$ in each cluster $V_i$, $i\in [k]$ and remove all edges not touching $U:=\bigcup_{i\in [k]} U_i$ from $G$.
Then we add a $k$-partite Erd\H{o}s graph (obtained from random graph) to $V(G)\setminus U$ so that the resulting graph $G'$ satisfy $\alpha^*_{\rm b}(G')=o( n)$ but $G^\prime-U$ is $C_k$-free.
Now every transversal copy of $C_k$ in $G'$ must contain a vertex in $U$ and $\bar{\delta}(G')\geq n/k-1$.
Since $|U|<n$, $G'$ does not have a transversal $C_{k}$-factor.

If we remove the $\alpha^*_{\rm b}$ condition in Lemma~\ref{33}, then the minimum degree threshold for transversal $C_{k}$-factor is asymptotically $\bar{\delta}(G)\geq (1+\frac1k)\frac n2+o(n)$, as determined recently by Ergemlidze and Molla in~\cite{ergemlidze2022transversal}.

Our proof of Lemma~\ref{33} is based on the absorption method, which will be given in Section \ref{77}.
Finally, there are some other minimum-degree-type results in blown-up graphs~\cite{ergemlidze2022transversal,johansson2000triangle} and in multi-partite graphs~\cite{1999Variants, KeevashMycroft, LoMarkstrom, Magyar2002TripartiteVO, MARTIN20084337}, but not with any randomness condition.

\section{Proof of the main theorem}

\begin{proof}[ Proof of Theorem \ref{1}]
	Given a positive integer $\Delta$ and a constant $\varepsilon>0$, we set $k=\lceil \frac{48}{\varepsilon} \rceil $, $\gamma=\frac{1}{4k\Delta^2}$ and $\eta=\frac{1}{4k\Delta}$. 
	Choose 
	\begin{equation}
		\frac{1}{n}\ll\alpha\ll\frac{1}{d}\ll\varepsilon,\frac{1}{\Delta} \nonumber
	\end{equation}
	and let  $G$ be an $n$-vertex graph with $\delta(G)\ge\varepsilon n$ and $\alpha^*(G)<\alpha n$. 
	Given any tree $T\in \mathcal{T}(n,\Delta)$,
	by Corollary \ref{2},  we proceed the proof by considering the following two cases.
	
	\noindent\textbf{Case 1.}
	$T$ has at least $\eta n$ pendant stars.
	
	We can easily pick a collection  of $\gamma n$ vertex-disjoint pendant stars in $T$ and label it as  $\mathcal{D}=\{D_1,\dots,D_{\gamma n}\}$  for convenience. Write $A:=\left\{ a_{1},\dots,a_{\gamma n} \right\}$ and $B:=\left\{ b_{1},\dots,b_{\gamma n} \right\}$  where $a_{i}$ and $b_{i}$ are the root and center of  $D_i$ respectively.
	Let $\mathcal{S}=\left\{ S_{1},\dots,S_{\gamma n} \right\}$ where each $S_{i}$ is obtained from $D_i$ by removing the root.	Let $T_{1}$ be a subtree of $T$ obtained by deleting vertices of $\mathcal{S}$.
	Note that $\vert T_{1}\vert \leq n-2\gamma n$. Moreover
	we claim that $V(G)$ can be  partitioned into $V_{1},V_{2},V_{3}$ of sizes $n_{1},n_{2},n_{3}$ respectively, such that
	\begin{equation}\label{100}
		d_{V_{i}}(v)\geq\frac{\varepsilon}{2}\vert V_{i} \vert{\rm\ for\ all}\ v \in V(G),\ i\in [3]
	\end{equation}
	where 
	\begin{equation}\label{98}
		n_{1}= \frac{d\vert T_{1}\vert+4\Delta d}{d-2\Delta},\quad
		n_{2}=\gamma n,\quad
		n_{3}=n-n_{1}-n_{2} .
	\end{equation}
	Choose a partition $\left\{ V_{1},V_{2},V_{3} \right\}$ of $V(G)$ uniformly at random where  $\vert V_{i} \vert=n_{i}$.
	For every $v\in V(G)$, let  $f^{i}_{v}=d_{V_i}(v)$ and note that $\mu_{v}^{i}:=\mathbb{E}[f_{v}^{i}]\geq\varepsilon n_{i}$. 
	Let $q$ be the probability that there exist $v\in V(G)$ and $i\in [3]$ which violates property $(\ref{100})$.
	Then by  the union bound and Chernoff’s inequality
	(see e.g.  \cite{janson2011random}, Theorem 2.1), 
	
	\begin{equation}
		q\leq 3n\exp\left(\frac{-(\mu_{v}^{i}/2)^2}{2\mu_{v}^{i}} \right)
		\leq 3n \exp\left(-\frac{\varepsilon n_{i}}{8} \right)=o(1) \nonumber
	\end{equation}
	for sufficiently large $n$. Therefore, with positive probability the randomly chosen partition $\left\{ V_{1},V_{2},V_{3} \right\}$ satisfies property $(\ref{100})$. Then we have
	\begin{align}\label{n3}
		n_{3}&\geq n-\gamma n-\frac{d(n-2\gamma n)+4\Delta d}{d-2\Delta}=\frac{(d+2\Delta)\gamma n-2\Delta n-4\Delta d}{d-2\Delta} \notag\\
		&\geq\frac{d\gamma n-2\Delta n}{2d}\geq\frac{\gamma n}{4}
		\geq \frac{6\Delta\alpha n}{\varepsilon},
	\end{align}
	where in the  penultimate inequality we use $d\gamma n-2\Delta n\geq\frac{d\gamma n}{2}$ bacause $\frac{1}{d}\ll\varepsilon,\frac{1}{\Delta}$ and
	the last inequality follows since $\alpha\ll\varepsilon,\frac{1}{\Delta}$.

	\begin{claim} \label{12}
		\rm{$G[V_{1}]$ is an $(n_1,d)$-expander}.
	\end{claim}
	\begin{proof}
		Let $m_{1}=\lceil \frac{\vert V_{1}\vert}{2d}\rceil$ and $m_{2}=\lceil \frac{\varepsilon \vert V_{1}\vert}{2d+2} \rceil$.	
		Since $\alpha^*(G)<\alpha n \leq m_{1}$, there is at least one edge between any two vertex-disjoint sets of size $m_1$ in $G[V_{1}]$.
		For $X \subseteq V_{1}$ with $1\le\vert X\vert\le m_{2}$, by (\ref{100}), we have 
		\begin{equation}\label{80}
			\vert N_{V_{1}}(X)\vert \ge\frac{\varepsilon}{2} \vert V_{1}\vert-\vert X \vert\ge d\vert X\vert.
		\end{equation}
		For $X\subseteq V_{1}$ with $m_{2}<\vert X\vert\ <m_{1} $, since $\alpha \ll \frac{1}{d}$, $ \varepsilon$, we can arbitrarily pick $Z\subseteq X$ with $\vert Z\vert=\alpha n$. As there is no edge between $Z$ and $V_{1}\backslash (Z\cup N_{V_{1}}(Z))$, we have $\vert V_{1}\backslash (Z\cup N_{V_{1}}(Z))\vert<\alpha n$, then $\vert N_{V_{1}}(Z)\vert>\vert V_{1}\vert -2\alpha n$. Thus 
		\begin{equation}\label{81}
			\vert N_{V_{1}}(X)\vert\ge\vert N_{V_{1}}(Z)\vert-(\vert X\vert-\vert Z\vert)
			\geq\vert V_{1}\vert-\vert X\vert-\alpha n
			\ge  d\vert X\vert.
		\end{equation}
		Together with $(\ref{80})$, $G[V_{1}]$ is an $(n_1,d)$-expander.
	\end{proof}
	
	Note that
	$\vert T_{1}\vert
	= n_{1}-4\Delta(\frac{n_{1}}{2d}+1)
	\leq n_{1}-4\Delta \lceil \frac{n_{1}}{2d} \rceil$, 
	where the first equality follows since $n_{1}= \frac{d\vert T_{1}\vert+4\Delta d}{d-2\Delta}$.
	Then by Theorem \ref{9}, there exists an embedding $f_{1}:V(T_1)\to V_1$.
	Let $L_{0}=V_{1}\backslash f_{1}(T_{1})$ and $L_{1}= f_{1}(A)$.
	Next, we will embed  the centers of pendant stars
	into $V_{2} \cup V_{3}$ and we do it in two steps.
	
	In the first step, we embed most of the vertices of $B$  into $V_{2}$. Consider the  bipartite graph $H$ on vertex sets $L_{1}$ and $V_{2}$, where
	$\vert L_{1} \vert=\vert V_{2}\vert=\gamma n$.
	Let $M$ be a maximum matching in $H$ and we claim that $\vert E(M)\vert\geq\gamma n-\alpha n$. Otherwise, since $\alpha^*(G)< \alpha n$,
	there is at least  one edge in $H-V(M)$, contrary to the maximality of $M$.
	Let $L_2=V(M)\cap V_2$ and $L_{3}=V_{2}\setminus L_{2}$.
	Without loss of generality, suppose we have embedded $B_{1}=\left\{ b_{1},\dots,b_{t}\right\}$ into $L_2$, where $t\geq\gamma n-\alpha n$.

	In the second step, we shall embed the  vertices of $B\setminus B_{1}$ into $V_{3}$. For every $u\in V(G)$, by (\ref{100}) and (\ref{n3}), we have $d_{V_{3}}(u)\geq\frac{\varepsilon}{2}\vert V_{3} \vert\geq\Delta \alpha n$.
	 Hence we can greedily embed $S_{t+1},\dots,S_{ \gamma n}$ into $G[V_{3}]$.
		Let $\mathcal{S}_1=\{S_1,\dots,S_t\}$.
	It follows that there exists an embedding $f_{2}: V(B_{1})\cup V(\mathcal{S}\setminus \mathcal{S}_{1})\rightarrow V_{2}\cup V_{3}$ such that $f_{2}(B_{1})=L_{2}$ and $f_{2}(\mathcal{S}\setminus \mathcal{S}_{1})\subseteq V_{3}$. 
	Let $L_{4}=V_{3}\setminus f_{2}(\mathcal{S}\setminus \mathcal{S}_{1})$
	and we have $\vert L_{4} \vert\geq\vert V_{3}\vert-\Delta \alpha n$.
	
	Now it remains to embed the leaves attached to the vertices of $ B_{1}$. 
	Consider the bipartite graph $Q$ on vertex sets $L_{2}$ and $L$,
	where $L=L_{0}\cup L_{3}\cup L_{4}$. Let  $m:=2\alpha n$ and $d:=\Delta-1$. Since $\alpha^*(G)<\alpha n< m$, there is at least one edge between any two disjoint vertex set of size $m$ in $Q$.
	For all $X\subseteq L_{2}$ with $1\leq \vert X\vert\leq m$, by (\ref{100}) and (\ref{n3}), we have $\vert N_{Q}(X)\cap L\vert \geq \vert N_{Q}(X)\cap L_{4}\vert \geq \frac{\varepsilon}{2}\vert V_{3} \vert-\Delta\alpha n\geq d\vert X \vert $.
	Moreover for each $u \in L$,
	we have $d_{ L_{2}}(u)\geq\frac{\varepsilon}{2}\vert V_{2} \vert-\alpha n\geq m$ due to  $\alpha\ll\varepsilon,\frac{1}{\Delta}$. Therefore by applying Lemma \ref{11} on $Q$, we obtain 
	 an embedding $f_{3}$ of $\mathcal{S}_{1}$  in $Q$ such that $f_{3}(u)=f_{2}(u)$ for every $u\in B_{1}$.
	In conclusion, it is clear that the map $f: V(T)\rightarrow V(G)$ defined by
	\begin{equation}
		f(u):=\left\{
		\begin{array}{cl}
			f_{1}(u)& {\rm if}\; u\in V(T_{1}) \\
			f_{2}(u)  &{\rm if}\; u\in V(\mathcal{S})\setminus  V(\mathcal{S}_{1}) \\
			f_{3}(u) &{\rm if}\;  u\in V(\mathcal{S}_{1}) \\
		\end{array} \right.\nonumber
	\end{equation}
	is an embedding of $T$ in $G$. 
	The proof of Case 1 is complete.

	\noindent\textbf{Case 2.}  $T$ has at least $\eta n$ vertex-disjoint caterpillars of length $k$.	
	
	A caterpillar in $T$ consists of a bare path in $T^{\prime}$ as the central path with some (possibly empty) leaves attached to the internal vertices of the central path, where $T^\prime$ is the subtree obtained by deleting the leaves of $T$ and we say that  the internal vertices attached with leaves  are branch vertices.
	Observe that $T$  either has a family of at least $\frac{\eta n}{2}$ caterpillars of length $k$ that have at least one leaf or a family of at least $\frac{\eta n}{2}$ bare paths of length $k$. Here, we will give a detailed proof for the first subcase and the second subcase can be derived by the same argument.

	Let $n^{\prime}=\frac{\eta n}{2}$ and
	 $k^{\prime}$ be an integer from $\{\frac{k}{2},\frac{k}{2}-1,\frac{k}{2}-2,\frac{k}{2}-3\}$ such that $k^{\prime}=2$ 
	(mod 4). 	
	It is easy to pick a collection of 
	$n^{\prime}$ vertex-disjoint caterpillars of length $k^{\prime}$ in $T$ such that one end of each caterpillar is adjacent to a branch vertex in $T$. 	
	Let $\mathcal{F}=\left\{ F_{1},\dots,F_{n^{\prime}} \right\}$ be such a collection
	and $\mathcal{P}=\left\{ P_{1},\dots,P_{n^{\prime}} \right\}$, where each $P_{i}$ is the central path of $F_{i}$.
	Write $S:=\left\{ s_{1},\dots,s_{n^{\prime}}\right\}$ and $W:=\left\{ w_{1},\dots,w_{n^{\prime}}\right\}$ where $s_{i}$ and $w_{i}$ are the ends of $F_{i}$ and  assume that the neighbor of $s_i$ in $P_i$ is a branch vertex,  $i\in[n^{\prime}]$.
	
	Let $T_{2}$ be a subforest of $T$ obtained by deleting the vertices of $\mathcal{F}$ except the ends of every caterpillar
	and note that $ \vert T_2\vert\leq n-n^\prime k^\prime$.
		In a similar way as Case 1, there exists a partition $\left\{ V_{1},V_{2},V_{3} \right\}$ of $V(G)$
	such that 
	\begin{equation}\label{60}
		d_{V_{i}}(u)\geq\frac{\varepsilon}{2}\vert V_{i} \vert{\rm\ for\ all}\ u \in V(G),\ i\in [3]
	\end{equation}
	where
	\begin{equation}\label{97}
		\vert V_{1}\vert=\vert T_{2} \vert+
		\frac{2\Delta\vert T_{2}\vert+4\Delta d}{d-2\Delta},\,
		\vert
		V_{2}\vert=n^{\prime}(k^{\prime}-1)-\frac{2\Delta\vert T_{2}\vert+4\Delta d}{d-2\Delta}, \,
		\vert V_{3}\vert= n-\vert V_{1}\vert-\vert V_{2}\vert.
	\end{equation}
	Then we have 
	\begin{equation}\label{190}
		\vert V_{3}\vert\geq n-(n-n^{\prime}k^{\prime})-n^{\prime}(k^{\prime}-1)
		= n^{\prime}
		\geq\frac{2\Delta\alpha n}{\varepsilon},
	\end{equation}
	where the first inequality follows since  $\vert T_2\vert\leq n-n^\prime k^\prime$ and  the last inequality follows since  $\alpha\ll \varepsilon,\frac{1}{\Delta}$.  
	
	Note that $\vert T_{2} \vert= \vert V_{1}\vert-4\Delta(\frac{\vert V_{1}\vert}{2d}+1)\leq
	\vert V_{1}\vert-4\Delta \lceil \frac{\vert V_{1}\vert}{2d}\rceil$, where the first equality follows since $\vert V_{1}\vert=\vert T_{2} \vert+
	\frac{2\Delta\vert T_{2}\vert+4\Delta d}{d-2\Delta}$. Then 
	Theorem \ref{9} implies that there exists an embedding $g_1$ of $T_{2}$ in $G[V_{1}]$.
	Let $L_{1}=V_{1}\backslash g_{1}(T_{2})$ and $V_{2}^{\prime}=V_2\cup L_1$.
	 It now remains to embed  $n^{\prime}$ caterpillars in $\mathcal{F}$.

	First, we shall  embed the internal vertices of $\mathcal{P}$ into $V_{2}^{\prime}$.	
	Randomly partition $V_2^{\prime}$ into  $X_1,\dots,X_{k^{\prime}-1}$ each of size $n^{\prime}$.
The union bound and	Chernoff's inequality  imply that, if $n$ is sufficiently large, then there exists a partition such that for every $u\in V(G)$ and $i\in[k^\prime-1]$, we have $d_{X_i}(u)\ge\frac{\varepsilon n^{\prime}}{4}$.
	Since $\alpha^*(G)< \alpha n\leq n^{\prime}$, there exists a matching $M_{1}$ between $g_{1}(S)$ and $X_{1}$ such that $t_1:=\vert E(M_{1})\vert>n^{\prime}-\alpha n $. 
	Similarly, there exists a matching $M_{2}$ between $g_1(W)$ and $X_{k^{\prime}-1}$ such that $t_2:=\vert E(M_{2})\vert> n^{\prime}-\alpha n $. 
	Let $S_1=\{s_1,\dots,s_{t_1}\}\subseteq S$ and $W_1=\{w_1,\dots,w_{t_2}\}\subseteq W$.
	Without loss of generality, we assume that $g_1(S_1)\subseteq V(M_{1})$ and $g_1(W_1)\subseteq V(M_{2})$.
	By the choice of $X_1,\dots, X_{k^\prime-1}$, for each $u\in g_1((S\cup W)G^\prime-Us (S_{1}\cup W_1))\subseteq V(G)$, we have
	\begin{equation}
		d_{X_2}(u)
		\geq\frac{\varepsilon n^{\prime}}{4}
		\ge 2\alpha n
		\geq(n^{\prime}-t_1)+(n^\prime-t_2),\nonumber
	\end{equation}
	and therefore we can greedily find a matching $M_3$ between $g_1(S\setminus S_{1})$ covering $g_1(S\setminus S_{1})$ and $X_2$ and a matching $M_4$ between $g_1(W\setminus W_{1})$ and $X_2$ covering $g_1(W\setminus W_{1})$, where $V(M_3)\cap V(M_4)=\emptyset$.	   
	Let $X_1^\prime:=(V(M_1)\cap X_1)\cup(V(M_3)\cap X_2)$, $X_{k^\prime-1}^\prime:=(V(M_2)\cap X_{k^\prime-1})\cup(V(M_4)\cap X_2)$, $X_2^\prime:=(X_2
	\setminus(V(M_3)\cup V(M_4)))
	\cup(X_1\setminus V(M_1))\cup(X_{k^\prime-1}\setminus V(M_2))$ and let
	$X_i^\prime:=X_i$ for $i\in [3,k^\prime-2]$.
	In this way, we obtain a new partition  $\left\{ X_{1}^\prime,\dots,X^\prime_{k^\prime-1} \right\}$ of $V_{2}^{\prime}$   such that there exist  perfect matchings between $X_1^\prime$ and $g_1(S)$ and between $X^\prime_{k^{\prime}-1}$ and $g_1(W)$.
	Let $X_{1}^\prime=\left\{ x_{1},\dots,x_{n^{\prime}}\right\}$ and $X^\prime_{k^{\prime}-1}=\left\{ y_{1},\dots,y_{n^{\prime}}\right\}$ where for each $i\in [n^{\prime}]$, $\{x_{i},g_1(s_{i})\}$ and $\{y_{i},g_1(w_{i})\}$ are edges of the above perfect matchings.
	
	Let $X^\prime_{0}=\left\{ z_{1},\dots,z_{n^{\prime}}\right\}$ be a new set of vertices disjoint from $V(G)$.		  
	Define an auxiliary graph $H$ with vertex set $V(H)=X^\prime_{0}\cup X^\prime_{2}\cup\cdots\cup X^\prime_{k^{\prime}-2}$. 
	For $2\leq i\leq k^{\prime}-3$, the edges of $H$ between $X^\prime_{i}$ and $X^\prime_{i+1}$ are identical to those of $G$. 
	For $v\in X^\prime_{2}$ and $z_{j}\in X^\prime_{0}$, $\left\{ v,z_{j}\right\}$ is an edge of $H$ if and only if $\left\{v,x_{j}\right\}$ is an edge of $G$. 
	Similarly, for $u\in X^\prime_{k^{\prime}-2}$ and $z_{j}\in X^\prime_{0}$, $\left\{ u,z_{j}\right\}$ is an edge of $H$ if and only if $\left\{ u,y_{j}\right\}$ is an edge of $G$. 
	Observe that if $H$ has a transversal $C_{k^{\prime}-2}$-factor, then $G$ has $n^{\prime}$ vertex-disjoint paths of length $k^{\prime}-2$ that connect $x_{i}$ and $y_{i}$. 
	Since we moved at most $2\alpha n$ vertices when we constructed the new partition, now we have that
	\begin{equation}\label{149}
		\bar{\delta}(H)
		\geq\frac{\varepsilon n^{\prime}}{4}-2\alpha n
		\geq \frac{\varepsilon n^{\prime}}{8}
		>\frac{2n^\prime}{k^\prime-2}. 
	\end{equation}
	Then  Lemma \ref{33} implies that $H$ contains a transversal $C_{k^{\prime}-2}$-factor.
	Together with the perfect matchings, we find an embedding $g_2$ of $V(\mathcal{P})$ to $V_2^{\prime}$ that connects $g_1(s_{i})$ and $g_1(w_{i})$ for each $i\in[n^{\prime}]$. 
	Now it suffices to embed the leaves of caterpillars into $V_{3}$.
	
	Let $I\subseteq V_2^{\prime}$ be the set of images of branch vertices in $V(\mathcal{F})$. 
	Since every vertex in $S$ is adjacent to a branch vertex in $F_i$, we have $X_{1}^\prime\subseteq I$. 
	Let $m:=\alpha n$ and $d:=\Delta$. 
	For all $X\subseteq I$ with $1\leq \vert X\vert\leq m$, by (\ref{60}) and (\ref{190}),
	we have $\vert N_{G}(X)\cap V_{3}\vert \geq \frac{\varepsilon}{2}\vert V_{3} \vert\geq d\vert X\vert$. 
	Moreover  for each $u \in V_{3}$, by (\ref{149}), we have $d_{I}(u) \geq d_{X^\prime_{1}}(u)\geq \frac{\varepsilon n^{\prime} }{8} \geq m$. Together with  
	the assumption that $\alpha^*(G)< m$,
	Lemma \ref{11} implies that there exists an embedding $g_{3}$ of $V(\mathcal{F})\backslash V(\mathcal{P})$ to $V_{3}$ that respects the edges between the branch vertices and the leaves of  caterpillars.
	It follows that the map $g: V(T)\rightarrow V(G)$ defined by
	\begin{equation}
		g(u):=\left\{
		\begin{array}{cl}
			g_{1}(u)& {\rm if}\; u\in V(T_{2}) \\
			g_{2}(u)  &{\rm if}\; u\in V(\mathcal{P}) \\
			g_{3}(u) &{\rm if}\;  u\in V(\mathcal{F})\backslash V(\mathcal{P}) \\
		\end{array} \right.\nonumber
	\end{equation}
	is an embedding of $T$ in $G$. This concludes the proof of the first subcase.
	
	As for the second subcase, we adopt a similar argument   and the main difference is that we   split $V(G)$ into two parts bacause the caterpillars have no leaves and it suffices to find $g_1$ and $g_2$.

\end{proof}

\section{ Transversal $C_{k}$-factor }
\subsection{Proof of  Lemma \ref{33}}
Following the typical absorption method, the main tasks are to ($i$) establish an absorbing set $R$ and ($ii$) find an almost perfect transversal $C_{k}$-tiling in $G-R$. For ($i$), 
we will introduce some related definitions.
\begin{definition}
	\rm {Let $G=(V_{1},\ldots,V_{k},E)$
		be  a spanning subgraph of the $n$-blow-up of $C_{k}$
		and $F$ be 
		a $k$-vertex graph.

		1. We say that a balanced subset $R\subseteq V(G)$ is a $\xi$-\textit{absorbing}\textit{ set} for some $\xi>0$ if for every balanced  subset $U\subseteq V(G)\backslash R$ with $\vert U\vert\leq\xi n$,
		$G[R\cup U]$ contains an $F$-factor which consists of transversal copies.
		
		2. Given a subset $S\subseteq V(G)$ of size $k$
		and an integer $t$, we say that a subset $A_{S}\subseteq V(G) \backslash S$ is an $(F,t)$-\textit{absorber} of $S$ if $\vert A_{S}\vert \leq kt$ and both $G[A_{S}]$ and $G[A_{S}\cup S] $ contain an $F$-factor.}
\end{definition}

Now we state the first crucial lemma, whose proof can be found in Section \ref{13}.

\begin{lemma}\label{3}
	{\rm(Absorbing Lemma)}. Given $k \in \mathbb{N}$ with $k\geq4$ and positive constants $\delta,\gamma$ with $\delta>\frac{2}{k}$ and $\gamma\leq\frac{\delta}{2}$,
	there exist $\alpha, \xi>0 $ such that the following holds for sufficiently large $n \in \mathbb{N} $. 
	Let $G=(V_{1},\ldots,V_{k},E)$ be  a spanning subgraph of the $n$-blow-up of $C_{k}$
	with $\bar{\delta}(G)\geq \delta n$ and $\alpha^*_{\rm b}(G)<\alpha n$. 
	Then there exists a $\xi$-absorbing set $R\subseteq V(G)$ of size at most $\gamma n$.

\end{lemma}
For ($ii$), Lemma $\ref{4}$ provides an
almost transversal $C_{k}$-tiling, whose proof will be given  in Section \ref{15}.

\begin{lemma} \label{4}
	{\rm(Almost perfect tiling)}.
	Given a positive integer $k\in 4\mathbb{N}$ and constants $\delta,\zeta$ with $\delta>\frac{2}{k}$, there exists $\alpha>0$ such that the following holds for sufficiently large $n \in \mathbb{N} $. Let $G=(V_{1},\ldots,V_{k},E)$ be  a spanning subgraph of the n-blow-up of $C_{k}$ with $\bar{\delta}(G)\geq \delta n$ and $\alpha^*_{\rm b}(G)<\alpha n$.  Then $G$ contains a transversal $C_{k}$-tiling covering all but at most $\zeta n$ vertices.
\end{lemma}
Now we are ready to prove  Lemma \ref{33} using Lemma \ref{3} and Lemma \ref{4}.

\begin{proof}[Proof of Lemma \ref{33}]
	Given $k \in 4\mathbb{N}$ and a constant $\delta$ with $\delta>\frac{2}{k}$, we set $\eta:=\delta-\frac{2}{k}$ and
	choose
	$\frac{1}{n}\ll\alpha\ll\zeta\ll\xi\ll\gamma\ll\eta,\delta$.  
	Let $G=(V_{1},\ldots,V_{k},E)$ be  a spanning subgraph of the $n$-blow-up of $C_{k}$ with $\bar{\delta}(G)\geq (\frac{2}{k}+\eta) n$ and $\alpha^*_{\rm b}(G)<\alpha n$. 
	
	By Lemma \ref{3} and 
	the choice that $\gamma \ll\eta,\delta$, there exists a $\xi$-absorbing set $R\subseteq V(G)$ of size at most $\gamma n$ for some $\xi>0$.
	Let $G^{\prime}:=G-R$ and note that $G^\prime$ is an $(n-\frac{\vert R\vert}{k})$-blow-up of $C_k$.
	Then we have
	\begin{equation}
		\bar{\delta}(G^{\prime})\geq \left(\frac{2}{k}+\eta\right) n-\frac{\gamma n}{k}\geq\left(\frac{2 }{k}+\frac{\eta}{2} \right)
		n.\nonumber
	\end{equation}
	Therefore by applying Lemma \ref{4} on $G^{\prime}$, we obtain a transversal $C_{k}$-tiling $\mathcal{M}$ that covers all but a set $U$ of at most $\zeta n$ vertices
	in $G^{\prime}$. Since $\zeta\ll\xi$, the absorbing property of $R$ implies that
	$G[R\cup U]$ contains a transversal $C_{k}$-factor, which together with  $\mathcal{M}$  forms
	a transversal $C_{k}$-factor in $G$.
	
\end{proof}

\subsection{Regularity}
The proof of Lemma \ref{4} is based on a standard application of the regularity method. 
We will introduce some basic definitions and properties.
Given a graph $G$ and a pair $(X,Y)$ of vertex-disjoint subsets in $V(G)$, the
\textit{density} of $(X,Y)$ is defined as
\begin{equation}
	d(X,Y)=\frac{e(X,Y)}{\vert X \vert \vert Y \vert}. \nonumber
\end{equation}
Given constants $\varepsilon,d>0$, we say that $(X,Y)$ is \textit{$(\varepsilon,d)$-regular} if $d(X,Y)\geq d$ and for all $X^{\prime} \subseteq X$, $Y^{\prime} \subseteq Y$ with $\vert X^{\prime} \vert \geq \varepsilon \vert X \vert$ and $\vert Y^{\prime} \vert \geq \varepsilon \vert Y \vert$, we have
\begin{equation}
	\vert d(X^{\prime},Y^{\prime})-d(X,Y)\vert \leq \varepsilon. \nonumber
\end{equation}
The following fact results from the
definition.	
\begin{fact}\label{22}
	\rm{Let $(X,Y)$ be an $(\varepsilon,d)$-regular pair and $B \subseteq Y$ with $\vert B\vert\ge\varepsilon\vert Y\vert.$
	Then all but at most $\varepsilon\vert X\vert$ vertices in $X$ have  at least $(d-\varepsilon)\vert B\vert$ neighbors in $B$}.
\end{fact}
We now state a degree form of the regularity lemma (see
\cite{1991Szemer}, Theorem 1.10). 

\begin{lemma}\label{19}
	{\rm(Degree form of  Regularity Lemma \cite{1991Szemer}).} 
	For every $\varepsilon>0$ there is an $N = N(\varepsilon)$
	such that the following holds for any real number $d\in(0,1]$ and $n\in \mathbb{N}$. Let $G=(V, E)$ be an $n$-vertex graph.
	Then there exists a partition $\mathcal{P}=V_{0}\cup V_{1}\cup\dots\cup V_{k}$  
	and a spanning subgraph
	$G^{\prime} \subseteq G$ with the following properties:	\\
	{\rm(\textit{a})} $\frac{1}{\varepsilon} \leq k \leq N$;\\
	{\rm(\textit{b})} $\vert V_{0} \vert \leq \varepsilon n$ and
	$\vert V_{1} \vert=\dots=\vert V_{k} \vert=m\leq \varepsilon n$;\\
	{\rm(\textit{c})} $d_{G^{\prime}}(v)\geq d_{G}(v)-(d+\varepsilon)n$ for all $v \in V(G)$;\\
	{\rm(\textit{d})} every $V_{i}$ is an independent set in $G^{\prime}$ for $i\in[k]$;\\
	{\rm(\textit{e})} every pair $(V_{i}, V_{j})$, $1 \leq i<j \leq k$  is $\varepsilon$-regular in $G^{\prime}$ with density $0$ or at least $d$.
\end{lemma}
A widely-used auxiliary graph accompanied with the regular partition is the reduced graph. The\textit{ reduced graph} $R_{d}$ of $\mathcal{P}$ is a graph defined on the vertex set $\left\{ V_{1},\ldots,V_{k} \right\}$ such that $V_{i}$
is adjacent to $V_{j}$ in $R_{d}$
if $(V_{i},V_{j})$ has density at least $d$ in $G^{\prime}$.
We use $d_{R}(V_{i})$ to denote the degree of $V_{i}$
in $R_{d}$ for each $i\in[k]$.

\begin{fact}\label{20}
	Given positive constants $d,\varepsilon$ and $\delta$, fix an $n$-vertex graph $G=(V,E)$ with $\delta(G)\geq \delta n$ and let $G^{\prime}$ and $\mathcal{P}$ be obtained by Lemma \ref{19}, and $R_{d}$ be given as above. Then for every $V_{i}\in V(R_{d})$, we have $d_{R}(V_{i})\geq(\delta-2\varepsilon-d)k$. 
\end{fact}

\subsection{Almost perfect tilings}\label{15}

Here we shall make use of the following result which provides a sufficient
condition for a transversal path among given sets.

	 \begin{proposition}\label{90}
		Given an integer $k\ge2$ and a positive  constant $\alpha\leq\frac{1}{2}$,  let $G=(V_{1},\ldots,V_{k},E)$ be a spanning subgraph of the $n$-blow-up of $C_k$ with  $\alpha^*_{\rm b}(G)<\alpha n$. 
		For any integers $i,j$ with $1\leq i<j\leq k$ and a collection of subsets $X_{s}\subseteq V_s$ with $s\in [i,j]$, if $\vert X_{i}\vert$, $\vert X_{j}\vert\geq\alpha n$ and $\vert X_{\ell}\vert\geq2\alpha n$ for $\ell\in [i+1,j-1]$ {\rm(}possibly empty{\rm)}, then there exists a transversal path   $x_{i}x_{i+1}\dots x_{j-1}x_{j}$   where $x_{s}\in X_{s}$ for $s\in [i,j]$.
	\end{proposition}
	
\begin{proof}[Proof of Proposition \ref{90}]
	Without loss of
	generality, we may take $i=1$, $j=k$
	for instance.
	Let $Z_{1}:=X_{1}$ and $Z_{2}:=N(Z_{1})\cap X_{2}$.
	By the fact that  $\alpha^*_{\rm b}(G)<\alpha n\leq \vert Z_{1} \vert, \vert X_{2} \vert$, it holds that $\vert Z_{2}\vert> \vert X_{2}\vert-\alpha n\geq\alpha n $.
	If $k=2$, then there exists an edge  $\left\{ x_{1},x_{2} \right\}$ between $Z_1$ and $Z_2$ and we are done.
	If $k>2$, then for each $s\in [k-1]$, there exist $Z_{s}\subseteq X_{s}$  of size larger than $\alpha n$ such that $Z_{s}$ is the set of neighbors of $Z_{s-1}$. 
	Since $\vert Z_{k-1}\vert$, $\vert X_{k} \vert\geq\alpha n$, there exists an edge  $\left\{ x_{k-1},x_{k} \right\}$
	between  $Z_{k-1}$ and $ X_{k}$.
	Therefore, we can find a transversal path $x_{1}x_{2}\dots x_{k-1}x_{k}$, where $x_{s}\in X_{s}$ for $s\in [k]$.
\end{proof}

Now we are ready to prove Lemma  \ref{4}.

\begin{proof}[ Proof of Lemma \ref{4}]
	Given $k\in 4\mathbb{N}$ and  $\delta,\zeta$ with $\delta>\frac{2}{k}$, we set $\eta:=\delta-\frac{2}{k}$ and choose $\frac{1}{n}\ll\alpha\ll\frac{1}{N_{0}}\ll\varepsilon\ll \zeta, \delta,\eta$. 
	Let $G=(V_{1},\ldots,V_{k},E)$ be  a spanning subgraph of the $n$-blow-up of $C_{k}$ with $\bar{\delta}(G)\geq \delta n$ and $\alpha^*_{\rm b}(G)<\alpha n$. 
	By applying Lemma \ref{19} on $G$ with $d:=\frac{\eta}{4}$, we obtain a partition $\mathcal{P}=\left\{ U_{0} \right\}\cup\left\{ U_{i,j}\subseteq V_{i}: i\in [k], j\in [N_{0}] \right\}$ that refines the partition $\{V_1,V_2,\dots,V_k\}$ of $G$ and a spanning subgraph $G^{\prime}$ of $G$ with properties $(a)-(e)$, where we write $m:=\vert U_{i,j} \vert$ for all $i\in [k]$,  $j\in [N_{0}]$.
	Let $R_{d}$ be the reduced graph defined on the vertex set $\left\{ U_{i,j}: i\in [k], j\in [N_{0}] \right\}$.
	For each $i\in[k]$, let $\mathcal{V}_{i}=\left\{ U_{i,j}: j\in [N_{0}] \right\}$ and note that $\{\mathcal{V}_i:i\in[k]\}$ is a partition of $R_d$.
	Then Fact \ref{20} implies that $\bar{\delta}(R_{d})\geq(\delta-\frac{\eta}{2})N_{0}=(\frac{2}{k}+\frac{\eta}{2})N_{0}$.     
	
	To obtain an almost perfect transversal $C_{k}$-tiling in $G$, we  define an auxiliary graph $H$	with $k$ vertices and two disjoint edges and then use it for embedding  copies of transversal $C_k$ (see Claim \ref{28}). 
	Now we will show that there exist $N_{0}$ vertex-disjoint copies of specific $H$  in $R_{d}$.

	\begin{claim}\label{105}
		For every $i\in [k-1]$,	$R_{d}[\mathcal{V}_{i},\mathcal{V}_{i+1}]$ has a matching of size  $\min\{N_0, 2\bar\delta(R_{d})\}$. 
	\end{claim}
	
	\begin{proof}
		
		Let $m=\min\{N_0, 2\bar\delta(R_{d})\}$ and
		without loss of generality, we may take $i=1$ for instance. Let $M$ be a maximum matching in  $R_{d}[\mathcal{V}_{1},\mathcal{V}_{2}]$  and
		assume for the contrary that
		$\vert E(M)\vert\leq m-1$. Let $U$ be the minimum vertex cover of    
		$R_{d}[\mathcal{V}_{1},\mathcal{V}_{2}]$.
		Then by K\"onig's theorem  (\cite{lovasz1986plummer}),  it holds that 
		$\vert U \vert=\vert E(M) \vert$.
		We write $A=U\cap\mathcal{V}_{1}$ and $B=U\cap\mathcal{V}_{2}$.
		By the pigeonhole principle, 
		we  get $\vert A\vert\leq\lfloor\frac{m-1}{2}\rfloor $ or $\vert B\vert\leq\lfloor\frac{m-1}{2}\rfloor $. 
		Suppose $\vert A\vert\leq\lfloor\frac{m-1}{2}\rfloor $.
		Since $m\leq N_0$, we have $\vert\mathcal{V}_{2}\backslash B \vert\neq0$.	Arbitrarily choose $u\in \mathcal{V}_{2}\backslash B$,
		then $u$ has no neighbor in $\mathcal{V}_{1}\backslash A$.
		Thus we have	$d_{\mathcal{V}_{1}}(u)=d_{A}(u)\leq\lfloor\frac{m-1}{2}\rfloor\leq \lfloor\frac{2\bar\delta(R_{d})-1}{2}\rfloor=  \bar{\delta}(R_{d})-1$, a contradiction.

	\end{proof}
	By Claim \ref{105}, for each  $i\in [\frac{k}{2}]$, there exists a matching $M_{i}$ of size $\min\{N_0, 2\bar\delta(R_{d})\}$ between $\mathcal{V}_{2i-1}$ and $\mathcal{V}_{2i}$.
	Let $A_{j}\subseteq\mathcal{V}_{j}$ be the vertices uncovered by matchings for $j\in[k]$.
	We pick a family $\mathcal{H}$ of $N_{0}$	vertex-disjoint copies of $H$ such that each copy contains two disjoint edges $e_{1}$ and $e_{2}$ where $e_{1}\in M_{2i-1}$	and $e_{2}\in M_{2i}$ for some $i\in [\frac{k}{4}]$ and exactly one vertex inside each $A_{j}$ for $j\in [k]\backslash \left\{4i-3,4i-2, 4i-1,4i\right\}$.
	Since
	\begin{equation}
		\sum_{i=1}^{\frac{k}{4}}\vert M_{2i-1} \vert=	\frac{k}{4}\cdot \min\{N_0, 2\bar\delta(R_{d})\} \geq N_{0},\nonumber
	\end{equation}
	and similarly $\sum_{i=1}^{k/4} \vert M_{2i}\vert\ge N_0$,
 we can greedily find $N_{0}$ vertex-disjoint copies of $H$  that together cover all vertices in $R_{d}$.

	\begin{claim}\label{28}
		For each copy of $H$ in $\mathcal{H}$, we can find a transversal $C_{k}$-tiling covering all but at most  $\frac{\zeta m}{2} $ vertices in the union of its clusters in $G$.
	\end{claim}
	
	\begin{proof}[Proof of Claim \ref{28}]
		Given a copy of $H$ ,
		without loss of generality, we may assume that $V(H)=\left\{U_{1,1},  U_{2,2},\ldots,U_{k,k}\right\}$ and $\left\{U_{1,1}, U_{2,2}\right\}$, $\left\{U_{3,3}, U_{4,4}\right\} \in E(H)$. 
		Therefore $(U_{1,1},U_{2,2})$ and $(U_{3,3},U_{4,4})$ are $(\varepsilon,d)$-regular in $G^\prime$. 
		Now	it suffices to show that for any $Z_{i}\subseteq U_{i,i}$ with $i\in[k]$	each of  size at least $\frac{\zeta m}{2k} $, there exists a copy of $C_{k}$ with exactly one vertex  inside each $Z_{i}$.
		
		Since $\vert Z_{i} \vert\geq \frac{\zeta m}{2k} \ge \varepsilon m$ for $i\in [4]$, Fact $\ref{22}$ implies that there exists a subset $Z_{2}^{\prime}\subseteq Z_{2}$ of size at least $\vert Z_{2} \vert-\varepsilon m$ such that every vertex in $Z_{2}^{\prime}$ has at least $(d-\varepsilon)\vert Z_{1} \vert$ neighbors inside $Z_{1}$ and a subset $Z_{3}^{\prime}\subseteq Z_{3}$ of size at least $\vert Z_{3} \vert-\varepsilon m$ such that every vertex in $Z_{3}^{\prime}$ has at least $(d-\varepsilon)\vert Z_{4} \vert$ neighbors inside $Z_{4}$ respectively.

		By the assumption $\alpha^*_{\rm b}(G)<\alpha n$ and the fact that $\vert Z_2^\prime\vert, \vert Z_3^\prime\vert\ge\frac{\zeta m}{2k}-\varepsilon m \ge\alpha n$, there is at least one edge between $Z_{2}^{\prime}$ and $Z_{3}^{\prime}$.
		Arbitrarily choose one edge $\left\{u,v\right\}$ with $u\in Z_2'$ and $v\in Z_3'$, 
		and let $Q_{1}=N(u)\cap Z_{1}$ and $Q_{2}=N(v)\cap Z_{4}$ respectively. 
		Then we have $\vert Q_{1} \vert,\vert Q_{2} \vert\geq
		(d-\varepsilon)\cdot\frac{\zeta m}{2k}\geq \alpha n$ and  note that $\vert Z_i\vert\ge\frac{\zeta m}{2k}\ge 2\alpha n$ for every $i\in [k]$.
		By applying Proposition \ref{90} with $X_{i}:=Q_{1}$ and $X_{j}:=Q_{2}$, we can obtain a transversal path of length $k-3$, where the ends in $Q_{1}$ and $Q_{2}$ are denoted by $u^{\prime}$ and $v^{\prime}$ respectively. Together with three edges $\left\{u^{\prime},u\right\}$, $\left\{u,v\right\}$, $\left\{v,v^{\prime}\right\}$, we can  construct a copy of transversal $C_{k}$ in $\cup_{i=1}^{k}Z_{i}$. 
		Thus we can obtain a transversal $C_{k}$-tiling covering all but at most $\frac{\zeta m}{2}$	vertices in $\cup_{i=1}^{k} U_{i,i}$ in $G$. 
	\end{proof} 
	
	This would finish the proof as  the union of these $C_{k}$-tilings taken over all copies of $H$ in $\mathcal{H}$ would leave at most
	\begin{equation}
		\vert U_{0}\vert+\vert \mathcal{H} \vert \cdot	\frac{\zeta m}{2}
		\leq \varepsilon n+\frac{\zeta n }{2}\leq\zeta n \nonumber
	\end{equation}
	vertices uncovered.
\end{proof}

\subsection{Building an absorbing set}\label{77}

A typical step in the absorption method for $F$-factor is to show that	every $k:=\vert V(F) \vert$-set  has polynomially many absorbers (see \cite{2203.02169}). 
However, it remains unclear  whether this property holds in our setting.
Instead, a new approach due to Nenadov and Pehova \cite{nenadov2020ramsey} guarantees an absorbing set provided
that  every $k$-set  has linearly many vertex-disjoint absorbers.
Since the host graph  in Lemma \ref{3} is  $k$-partite, we aim to show that  every transversal $k$-set has linearly many vertex-disjoint absorbers. 
For this, we shall make use of the \textit{lattice-based absorbing method} developed by Han \cite{Han2017Decision}.

\subsubsection{ Finding absorbers}
To illustrate the lattice-based absorbing method, we introduce some definitions.
Let  $G, F$ be given as aforementioned and $m, t$ be positive integers.
Then we say that two vertices $u,v\in V(G)$ are $(F, m, t)$-\textit{reachable} (in $G$) if for any set $W$ of $m$ vertices, there is a set $S\subseteq V(G)\backslash W$ of size at most $kt-1$ such that both $G[\left\{u \right\}\cup S]$ and $G[\left\{v \right\}\cup S]$ have $F$-factors, where we call such $S$ an $F$-$\textit{connector}$ for $u, v$. Moreover, a set $U\subseteq V(G)$ is $(F, m, t)$-$\textit{closed}$ if every two vertices $u, v$ in $U$ are $(F, m, t)$-reachable, where the corresponding $F$-connector for $u, v$ may not be included in $U$.

The following result builds a sufficient condition to ensure that every transversal $k$-set has linearly many vertex-disjoint absorbers.
\begin{lemma}\label{5}
	Given  $k \in \mathbb{N}$ with $k\geq4$  and a constant $\delta>\frac{2}{k}$,
	there exist $\alpha,\beta>0 $ such that the following holds for sufficiently large $n \in \mathbb{N} $.
	Let $G=(V_{1},\ldots,V_{k},E)$ be  a spanning subgraph of the n-blow-up of $C_{k}$ with $\bar{\delta}(G)\geq \delta n$ and $\alpha^*_{\rm b}(G)<\alpha n$. 
	Then every transversal $k$-set in $G$ has at least $ \frac{\beta n-k}{4k^2}$
	vertex-disjoint $(C_{k},2k)$-absorbers. 
\end{lemma}

\subsubsection{Proof of Lemma \ref{3}} \label{13}
In order to prove the existence of an absorbing set,
we 
introduce a  notion of $F$-fan.

\begin{definition}
	{\rm(\cite{han2021ramsey})
		For a vertex $v\in V(G)$ and a $k$-vertex graph $F$, an \textit{$F$-fan} $\mathcal{F}_{v}$ at $v$
		in $V(G)$ is a collection of pairwise disjoint sets $S\subseteq  V(G)\backslash \left\{v \right\}$ such that for each $S\in \mathcal{F}_{v}$
		we have that $\vert S\vert=k-1$ and $\left\{v \right\}\cup S$ spans a copy of $F$.
		
	}
\end{definition}

To build an absorbing structure, we shall make use of bipartite templates as follows, which was introduced by Montgomery \cite{montgomery2019spanning}.

\begin{lemma}\label{31}
	Let $\beta>0$. There exists $m_{0}$ such that the following holds for every $m\geq m_{0}$.
	There exists a bipartite graph $B_{m}$ with vertex classes $X_{m}\cup Y_{m} $ and $Z_{m}$ and maximum degree 40, such that $\vert X_{m}\vert=m+\beta m$, $\vert Y_{m}\vert=2m$ and $\vert Z_{m}\vert=3m$, and for every subset $X^{\prime}_{m}\subseteq X_{m}$ of size $\vert X^{\prime}_{m}\vert=m$, the induced graph $B[X^{\prime}_{m}\cup Y_{m},Z_{m}]$ contains a perfect matching.
\end{lemma}

\begin{proof}[Proof of Lemma \ref{3} ]
	Given $k \in \mathbb{N}$ with $k\geq4$ and positive constants $\delta>\frac{2}{k}$ and $\gamma\leq\frac{\delta}{2}$, we choose $\frac{1}{n}\ll\alpha\ll\xi\ll\beta\ll\delta,\gamma, \frac{1}{k}$.
	Let $G=(V_{1},\ldots,V_{k},E)$ be a spanning subgraph of the $n$-blow-up of $C_{k}$ with $\bar{\delta}(G)\geq \delta n$ and $\alpha^*_{\rm b}(G)<\alpha n$.  
	Lemma \ref{5}  implies that  every transversal $k$-set  in $G$ has at least $\frac{\beta n-k}{4k^2}$ vertex-disjoint $(C_{k},2k) $-absorbers. 
	Let $\tau:=\frac{\beta}{8k^2}$.
	Then for every $v\in V(G)$, there is a $C_k$-fan $\mathcal{F}_{v}$ in $V(G)$ of size at least $ \frac{\beta n-k}{4k^2}\ge\tau n $.
	Now it suffices to find a $\xi$-absorbing set $R$ for some $\xi >0$ such that $\vert R \vert\leq\tau n\leq\gamma n$.

	Let $q=\frac{\tau}{1000k^3}$ and $\beta^{\prime}=\frac{q^{k-1}\tau}{2k}$.
	For $i\in[k]$, let $X_{i}\subseteq  V_{i} $ be a set of size $qn$ chosen uniformly at random.
	For every $v\in V(G)$, let $f_{v}$ denote the number of the sets from $\mathcal{F}_{v}$ that lie inside $\cup_{i=1}^{k} X_{i}$. Note that $\mu:=\mathbb{E}[f_{v}]= q^{k-1}\vert\mathcal{F}_{v} \vert\geq q^{k-1}\tau n$. By the union bound and Chernoff's inequality, we have
	\begin{equation}
		\mathbb{P}\left[\,{\rm there}\ {\rm is}\ v\in V(G)\ {\rm with}\ f_{v}<\frac{\mu}{2}\,\right]\leq kn\, {\rm exp}\left(\frac{-(\mu/2)^2}{2\mu}\right)\leq kn\, {\rm exp}\left(-\frac{q^{k-1}\tau}{8}n\right)=o(1). \nonumber
	\end{equation}
	Therefore, as $n$ is sufficiently large, there exist $X_{i}\subseteq V_{i} $ with  $\vert X_{i} \vert= qn$	such that for each $v\in  V(G)$, there	is a subfamily $\mathcal{F}_{v}^{\prime}$ of at least $\frac{q^{k-1}\tau n}{2}= k\beta^{\prime}n$ sets from $\mathcal{F}_{v}$ contained  in $\cup_{i=1}^{k} X_{i}$.

	Let $m=\vert X_{i} \vert/(1+\beta^{\prime})$ and note that $m$ is linear in $n$. 
	Let $\left\{ I_{i}  \right\}_{i\in [k]}$ be a partition of $[3km]$ with each $\vert I_{i} \vert=3m$.
	For $i\in [k]$,	arbitrarily choose $k$ vertex-disjoint subsets $Y_{i}$,  $ Z_{i,j}$ for $j\in[k]\backslash \left\{ i\right\}$ in $V_{i}\backslash X_{i}$	with $\vert Y_{i}\vert=2m$ and $\vert Z_{i,j}\vert=3m$.
	Let $X=\cup_{i=1}^{k} X_{i}$, $Y=\cup_{i=1}^{k} Y_{i}$	and $ Z=\cup Z_{i,j}$.
	Then we have $\vert X\vert=(1+\beta^{\prime})km$, $\vert Y\vert= 2km$ and $\vert Z \vert=3k(k-1)m$. 
	For each $j\in[k]$, we partition $\cup_{i\in[k]\backslash\left\{ j \right\}}Z_{i,j}$ into a family $\mathcal{Z}_j$ of $3m$ transversal $(k-1)$-sets and take an arbitrary bijection $\phi_{j}:\mathcal{Z}_j\rightarrow I_{j}$.
	Moreover, we define a function $\varphi$ on $[3km]$ such that $\varphi(x):=\phi^{-1}_{j}(x)$ if $x\in I_{j}$.
	Let  $ T_{i}$ be the bipartite graph obtained by Lemma \ref{31} with vertex classes $X_{i}\cup Y_{i}$ and $I_{i}$, 
	and let $T=\cup_{i=1}^{k} T_{i}$. Then $T$ 
	is a bipartite graph between $X\cup Y$ and $[3km]$ with $\Delta(T)\leq40$.

	We claim that there exists a family $\left\{ A_{e} \right\}_{e\in E(T)}$ of pairwise vertex-disjoint subsets in $V(G)\backslash (X\cup Y\cup Z)$ such that for every $e=\left\{ w_{1},w_{2} \right\}\in E(T)$ with $w_{1}\in X\cup Y$ and $w_{2}\in [3km]$, the set $A_{e}$ is a $(C_{k}, 2k)$-absorber for the transversal $k$-set $\left\{ w_{1} \right\}\cup\varphi(w_{2})$. 
	Indeed otherwise, there exists an edge $e^{\prime}\in E(T)$ without such a subset.
	Recall that $m=\frac{qn}{1+\beta^{\prime}}$ and $\Delta(T)\leq40$,  then we have
	\begin{equation}
		\vert X\vert+\vert Y\vert+\vert Z\vert+\left\vert\bigcup_{e\in E(T)\backslash\left\{ e^{\prime}\right\}} A_{e} \right\vert \leq 4km+3km(k-1)+2k^2\vert E(T) \vert\leq 4k^2m+80k^2\cdot 3km\leq\frac{\tau n}{2}. \nonumber
	\end{equation}
	Since every transversal $k$-set  has at least $\tau n$ vertex-disjoint $(C_{k}, 2k)$-absorbers in $G$, we can choose one  in $V(G)\backslash (X \cup Y \cup Z\cup \bigcup_{e\in E(T)\backslash\left\{ e^{\prime}\right\}}A_{e})$ as the subset $A_{e^{\prime}}$,	a contradiction.

	Let $R=X \cup Y \cup Z\cup \bigcup_{e\in E(T)}A_{e}$. Then $\vert R \vert \leq \tau n$ and
	we claim that $R$ is a $\xi$-absorbing set in $G$.
	Indeed, for an arbitrary  balanced  subset $U\subseteq V(G)\backslash R$ with $\vert U\vert\leq\xi n$, we shall verify that $G[R\cup U]$ admits a $C_k$-factor. 
	Note that if there exist $Q_{i}\subseteq X_{i}$ with $\vert Q_{i} \vert=\beta^{\prime}m$ for $i\in[k]$ and   a transversal $C_{k}$-factor in $G[\bigcup_{i=1}^{k} Q_{i}\cup U]$,	then $G[R\cup U]$ contains a transversal $C_{k}$-factor.
	In fact	by setting $X_{i}^{\prime}=X_{i}\backslash Q_{i}$, Lemma \ref{31} implies that there is a perfect matching $M$ in $T$ between $ \bigcup_{i=1}^{k} X_{i}^{\prime}\cup Y$ and $[3km]$.
	For each edge $e=\left\{ w_{1},w_{2} \right\}\in M$ take a transversal $C_{k}$-factor in $G[\left\{ w_{1} \right\}\cup\varphi(w_{2})\cup A_{e}]$  and for each $e^\prime\in E(T)\backslash M$ take a transversal $C_{k}$-factor in $G[A_{e^\prime}]$, which  forms a transversal $C_{k}$-factor of $G[R\backslash \cup_{i=1}^{k} Q_{i}]$.
	Thus together with the above assumption, we can obtain a transversal $C_{k}$-factor in $G[R\cup U]$.
	
	Hence, it suffices to find the desired $Q_{i}$ as above.
	Recall that every $v\in V(G)$ has a subfamily $\mathcal{F}_{v}^{\prime}$ of at least $k\beta^{\prime}n$ sets in $\cup_{i=1}^{k} X_{i}$.
	By the choice that $\xi\ll\beta,\frac{1}{k}$ and thus $\vert U\vert\leq\xi n\leq \beta^{\prime}n$, one can greedily find a family $\mathcal{C}_{1}$ of vertex-disjoint copies of transversal $C_{k}$ covering $U$ with vertices in $\cup_{i=1}^{k}X_{i}$.
	Let $Q_{i1}:=X_{i}\cap V(\mathcal{C}_{1})$.
	Then we have $\vert Q_{i1}\vert=\frac{k-1}{k}\vert U  \vert $ and $\mathcal{C}_1$ is a transversal $C_{k}$-factor in $G[\bigcup_{i=1}^{k} Q_{i1}\cup U]$.
	Moreover, one can greedily pick a family $\mathcal{C}_{2}$ of		  $\beta^{\prime}m-\frac{k-1}{k}\vert U  \vert $ vertex-disjoint copies of transversal $C_{k}$ in $G[X\backslash V(\mathcal{C}_{1})]$.
	This is possible since   every vertex $v$ has at least  $k\beta^{\prime}n-(k-1)\vert U\vert\geq k(\beta^{\prime}m-\frac{(k-1)}{k}\vert U  \vert)$ sets from $\mathcal{F}_v$ that are disjoint from   $V(\mathcal{C}_{1})$ .
	Let $Q_{i2}:=X_{i}\cap V(\mathcal{C}_{2})$ and $Q_{i}:=Q_{i1}\cup Q_{i2}$.
	Then $Q_i$ is desired because $\mathcal{C}_{1}\cup \mathcal{C}_{2}$ is indeed a transversal $C_k$-factor of $G[\bigcup_{i=1}^k Q_i\cup U]$.  
	This completes the entire proof.
\end{proof}

\subsubsection{Proof of Lemma \ref{5}}

\begin{proof}[Proof of Lemma \ref{5} ]

	Given $k \in \mathbb{N}$ with $k\geq4$ and $\delta>\frac{2}{k}$, we set $\eta:=\delta-\frac{2}{k}$ and choose 
	$\frac{1}{n}\ll\alpha\ll\beta\ll\delta,\eta$. 
	Let $G=(V_{1},\ldots,V_{k},E)$ be  a spanning subgraph of the $n$-blow-up of $C_{k}$ with $\bar{\delta}(G)\geq \delta n$ and $\alpha^*_{\rm b}(G)<\alpha n$. 
	\begin{claim}\label{96}
		For each $i\in[k]$, $V_{i}$ is
		$(C_{k},\beta n,2)$-closed.
	\end{claim}
	\begin{proof}[Proof of Claim \ref{96} ]
		
		Without loss of generality, we may assume that $i=1$.
		For any two vertices $u,v\in V_{1}$, since $\bar{\delta}(G)\geq \delta n$,
		we can choose four vertex-disjoint sets $D_{1}$, $D_{2}\subseteq V_{2}$ and $D_{3}, D_{4}\subseteq V_{k}$  each of size at least $\frac{\delta n}{2}$
		such that $D_{1}\subseteq N_{V_{2}}(u)$, $D_{2}\subseteq N_{V_{2}}(v)$, $D_{3}\subseteq N_{V_{k}}(u)$ and $D_{4}\subseteq N_{V_{k}}(v)$
		respectively.
		Given any vertex set $W\subseteq V(G)\backslash\left\{u,v \right\} $ of size at most  $\beta n$, 
		let $V^{\prime}_{i}=V_{i}\backslash W$  and $D^{\prime}_{j}=D_{j}\backslash W$ for  $i\in[k]$ and $j\in [4]$. Note that $\vert V^{\prime}_{i}\vert\geq n-\beta n\geq10\alpha n$ and $\vert D^{\prime}_{j}\vert\geq\frac{\delta n}{2}-\beta n\geq\alpha n$.

		Since $\alpha^*_{\rm b}(G)<\alpha n\leq\vert D_{j}^{\prime}\vert, \vert  V^{\prime}_{1}\vert$,  
		there exist subsets  $S_{j}\subseteq V^{\prime}_{1}$ of size at least $\vert V^{\prime}_{1}\vert-\alpha n$ such that  every vertex in
		$S_{j}$ has  at least one
		neighbor inside $D_{j}^{\prime}$. By the fact that $\vert S_{j}\vert\geq\vert V^{\prime}_{1}\vert-\alpha n>\frac{3}{4}\vert V^{\prime}_{1}\vert$, 
		we have that $S:=\bigcap_{j=1}^{4}S_{j}\neq\emptyset$.
		Arbitrarily choose $x\in S$, and therefore there exist vertices $y_{1}$, $y_{2}$, $y_{3}$, $y_{4}$ satisfying
		$y_{i}\in N_{D_{i}^{\prime}}(x)$ for $i\in [k]$.
		Note that $\vert N_{V^{\prime}_{3}}(y_{1})\vert$, $\vert N_{V^{\prime}_{k-1}}(y_{3})\vert\geq\delta n-\beta n\geq\alpha n$
		and $\vert V^{\prime}_{i}\vert\geq  10\alpha n$ for $i\in[k]$.  	When	$k>4$,	by applying
		Proposition \ref{90} with $X_{i}:=N_{V^{\prime}_{3}}(y_{1})$ and $X_{j}:=N_{V^{\prime}_{k-1}}(y_{3})$, we can obtain a transversal cycle $C_{1}$ passing through $u, y_{1}, y_{3}$.  When $k=4$, since  $\vert N_{V^{\prime}_{3}}(y_{1})\vert$, $\vert N_{V^{\prime}_{3}}(y_{3})\vert\geq\delta n-\beta n\geq(\frac{1}{2}+\eta-\beta) n>\frac{n}{2}$, we can easily find a common neighbor of $y_{1}, y_{3}$ and thus obtain a transversal cycle $C_{1}$ passing through $u, y_{1}, y_{3}$. 
	Similarly, we can obtain a transversal cycle $C_{2}$ in $V(G)\backslash(W\cup V(C_{1}))$ that passes through $v, y_{2}, y_{4}$.
 	In fact, the set $\left\{ x \right\}\cup\left(V(C_{1})\backslash\left\{ u \right\}\right)\cup \left(V(C_{2})\backslash\left\{ v \right\}\right)$ is a $C_{k}$-connector for $u$, $v$. Therefore by definition, $u$ is $(C_{k},\beta n,2)$-reachable to $v$.
		
	\end{proof}
	
	For every transversal $k$-subset $S\subseteq V(G)$, we greedily find as many pairwise disjoint $(C_{k},2k)$-absorbers for $S$ as possible.
	For convenience,
	we write $S=\left\{ s_{1}, s_{2},\ldots,s_{k} \right\}$ where $s_{i}\in V_{i}$ for $i\in[k] $.
	Let $\mathcal{A}=\left\{ A_{1}, A_{2},\ldots,A_{\ell} \right\} $ be a maximal family of such absorbers. Suppose to the contrary that $\ell< \frac{\beta n-k}{4k^2}$. Since each $A_{j}$ has size at most $2k^2$, we have $\left\vert \bigcup_{j=1}^{\ell}A_{j}\right\vert< \frac{\beta n-k}{2}$.

	Since $\alpha\ll\beta\ll\delta$,  we can easily  find a copy  $T$ of transversal $C_{k}$ in $V(G)\backslash(\bigcup_{j=1}^{\ell}A_{j}\cup S)$ and write $T=\left\{ t_{1}, t_{2},\ldots,t_{k} \right\}$ where  $t_{i}\in V_{i}$ for $i\in[k] $.
	By the closedness of $V_{i}$, we can pick a collection $\left\{ I_{1}, I_{2},\ldots,I_{k} \right\}$ of vertex-disjoint subsets in $V(G)\backslash(\bigcup_{j=1}^{\ell}A_{j}\cup S\cup T)$ such that each $I_{i}$ is a $C_{k}$-connector for $s_{i}$, $t_{i}$ with $\vert I_{i} \vert\leq 2k-1$. In fact, for any $1\leq k^{\prime}\leq k$, we have
	\begin{equation}
		\left\vert \bigcup_{j=1}^{\ell}A_{j}\cup\left(\bigcup_{i=1}^{k^{\prime}}I_{i}\right)\cup S\cup T\right\vert\leq\frac{\beta n-k}{2}+k(2k-1)+2k<\beta n.  \nonumber
	\end{equation}
	Therefore,
	we can choose such $I_{i}$ one by one  because $s_{i}$ and $t_{i}$ are $(C_{k},\beta n,2)$-reachable.
	At this point, it is easy to verify that $\bigcup_{i=1}^{k} I_{i}\cup T$
	is actually a $(C_{k},2k)$-absorber for $S$, contrary to the maximality of $\ell$.

\end{proof}

\bibliographystyle{abbrv}
\bibliography{ref}

\end{document}